\documentclass[article,12pt]{amsart}

\usepackage{amsfonts}
\usepackage{amsmath}
\usepackage{amssymb}
\usepackage{mathrsfs}
\usepackage{graphicx}
\usepackage{color}
\usepackage{epsfig}
\usepackage{enumerate}

 \usepackage[foot]{amsaddr}

\numberwithin{equation}{section}
\theoremstyle{plain}
\newtheorem{thm}{Theorem}[section]
\newtheorem{rem}{Remark}[section]
\newtheorem{prop}{Proposition}[section]

\newtheorem{lem}{Lemma}[section]

\newcommand{\dE}{\mathbb{E}}
\newcommand{\dR}{\mathbb{R}}

\newcommand{\dP}{\mathbb{P}}
\newcommand{\dC}{\mathbb{C}}
\newcommand{\dN}{\mathbb{N}}
\newcommand{\dZ}{\mathbb{Z}}
\newcommand{\cA}{\mathcal{A}}
\newcommand{\cB}{\mathcal{B}}
\newcommand{\cC}{\mathcal{C}}
\newcommand{\cD}{\mathcal{D}}

\newcommand{\cN}{\mathcal{N}}
\newcommand{\cK}{\mathcal{K}}
\newcommand{\cS}{\mathcal{S}}
\newcommand{\rI}{\mathrm{I}}
\newcommand{\cF}{\mathcal{F}}
\newcommand{\cG}{\mathcal{G}}

\newcommand{\cH}{\mathcal{H}}

\newcommand{\veps}{\varepsilon}

\newcommand{\ind}{\mbox{1}\kern-.25em \mbox{I}}
\font\calcal=cmsy10 scaled\magstep1
\def\build#1_#2^#3{\mathrel{\mathop{\kern 0pt#1}\limits_{#2}^{#3}}}
\def\liml{\build{\longrightarrow}_{}^{{\mbox{\calcal L}}}}
\def\limp{\build{\longrightarrow}_{}^{{\mbox{\calcal P}}}}

\def\videbox{\mathbin{\vbox{\hrule\hbox{\vrule height1.4ex \kern.6em\vrule height1.4ex}\hrule}}}
\def\demend{\hfill $\videbox$\\}

\newcommand{\indicatrice}{\mathchoice{\rm 1\mskip-4mu l}{\rm 1\mskip-4mu l}{\rm 1\mskip-4.5mu l} {\rm 1\mskip-5mu l}}

\newcommand{\Hm}[1]{\leavevmode{\marginpar{\tiny%
$\hbox to 0mm{\hspace*{-0.5mm}$\leftarrow$\hss}%
\vcenter{\vrule depth 0.1mm height 0.1mm width \the\marginparwidth}%
\hbox to 0mm{\hss$\rightarrow$\hspace*{-0.5mm}}$\\\relax\raggedright
#1}}}
\usepackage{color}

\newcommand{\red}[1]{\textcolor{red}{#1}}

\begin{document}
\title[On the number of descents in a random permutation]
{Sharp large deviations and concentration inequalities for the number of descents in a random permutation}
\author{Bernard Bercu}
\thanks{The corresponding author is Adrien Richou, email address: adrien.richou@math.u-bordeaux.fr}
\address{Universit\'e de Bordeaux, Institut de Math\'ematiques de Bordeaux,
UMR CNRS 5251, 351 Cours de la Lib\'eration, 33405 Talence cedex, France.}
\email{bernard.bercu@math.u-bordeaux.fr}
\author{Michel Bonnefont}
\email{michel.bonnefont@math.u-bordeaux.fr}
\author{Adrien Richou}
\email{adrien.richou@math.u-bordeaux.fr}
\date{\today}

\begin{abstract}
The goal of this paper is to go further in the analysis of the behavior of the number of descents in a random permutation. Via 
two different approaches relying on a suitable martingale decomposition or on the Irwin-Hall distribution,  we prove that the number of descents 
satisfies a sharp large deviation principle. A very precise concentration inequality involving the rate function in the large deviation principle
is also provided.
\end{abstract}


\maketitle

\vspace{1ex}


\section{Introduction}
\label{S-I}


Let $\cS_n$ be the symmetric group of permutations on the set of integers $\{1,\ldots,n\}$ where $n \geq 1$. 
A permutation $\pi_{{ n}} \in \cS_n$ is said to have a descent at position $k \in \{1,\ldots,n-1\}$ if  $\pi_{ n}(k)>\pi_{ n}(k+1)$. 
Denote by $D_n=D_n(\pi_{ n})$ the random variable counting the number of descents of a permutation 
$\pi_{ n}$ chosen uniformly at random from $\cS_n$. We clearly have $D_1=0$ and for all $n \geq 2$,
\begin{equation}
\label{DEFDN}
D_{n} = \sum_{k=1}^{n-1} \rI_{\{\pi_{ n}(k)>\pi_{ n}(k+1)\}}.
\end{equation}
A host of results are available on the asymptotic behavior of the sequence $(D_n)$.
More precisely, we can find in B\'ona \cite{Bona2012} that for all $n \geq 2$,
\vspace{1ex}
$$
\dE[D_n]= \frac{n-1}{2} \hspace{1cm}\text{and}\hspace{1cm}
\text{Var}(D_n)=\frac{n+1}{12}.
$$

\noindent
{ 
In addition, it is possible to get a connection with generalized P\'olya's urn with two colors also known as Friedman's urn, see \cite{Friedman1949} and Remark \ref{REM-Urn} below. In particular, for this construction, we have by Corollary 5.2 in \cite{Freedman1965} the almost sure convergence 
\begin{equation}
\label{ASCVGDN}
\lim_{n \rightarrow \infty}\frac{D_n}{n}=\frac{1}{2} \hspace{1cm} \text{a.s.}
\end{equation}
Following the approach of Tanny \cite{Tanny1973}, see Section \ref{sub-Tanny} below, it is also possible to construct a different sequence $(D_n)$ with same marginal distribution using a sequence of independent random variables sharing the same uniform distribution on $[0,1]$. For this construction, we directly obtain the same almost sure convergence \eqref{ASCVGDN}, as noticed by Gnedin and Olshanski \cite{Gnedin2006}, Section 7.3. Nevertheless, the distribution of the process $(D_n)$ does not correspond to the one investigated in Section \ref{S-MG}.
}

\ \vspace{-1ex} \\
Four different approaches have been reported by Chatterjee and Diaconis \cite{Chatterjee2017} to establish the asymptotic normality
\begin{equation}
\label{ANDN}
\sqrt{n}\Bigl(\frac{D_n}{n} - \frac{1}{2}\Bigr)
\liml \cN \Bigl(0, \frac{1}{12}\Bigr).
\end{equation}
We also refer the reader to the recent contribution of Garet \cite{Garet2021} relying on the method of moments as well as to the recent proof by \"{O}zdemir \cite{Ozdemir2021} using a rather complicated martingale approach. Furthermore, denote by $L_n$ the
number of leaves in a random recursive tree of size $n$. 
It is well-known \cite{Zhang2015} that $L_{n+1}=D_n+ 1$. Hence, it has been proven by Bryc, Minda and Sethuraman \cite{Bryc2009} that
the sequence $(D_n/n)$ satisfies a large deviation principle (LDP) with good rate function given by
\begin{equation}
\label{DEFI}
I(x)=\sup_{ t \in \dR} \bigl\{xt - L(t) \bigr\}
\end{equation}
where the asymptotic cumulant generating function is such as
\begin{equation}\label{DEFL}
L(t)=\log \Bigl(\frac{\exp(t)-1}{t}\Bigr).
\end{equation}
The purpose of this paper is to go further in the analysis of the behavior of the number of descents by proving a sharp large deviation principle (SLDP) for the sequence $(D_n)$. We shall also establish a sharp concentration inequality involving the rate function $I$ given by \eqref{DEFI}.

\ \vspace{-1ex} \\
To be more precise, we propose two different approaches that lead us to a SLDP and a concentration inequality for the sequence
$(D_n)$. The first one relies on a martingale approach while the second one use a miraculous link between the distribution of $(D_n)$ and the Irwin–Hall distribution, as pointed out by Tanny \cite{Tanny1973}. On the one hand, the second method is more direct and simpler in order to establish our results. On the other hand, the first approach is much more general and we are strongly convinced that it can be extended to other statistics on random permutations that share the same kind of iterative structure such as the number of alternating runs \cite{Bona2012, Stanley2008} or the length of the longest alternating subsequence in a random permutation \cite{Houdre2010, Widom2006}. Moreover, we have intentionally kept these two strategies of proof in the manuscript in order to highlight that the martingale approach is as efficient and powerful as the direct method in terms of results.


\ \vspace{-1ex} \\
The paper is organized as follows. Section \ref{S-MG} is devoted to our martingale approach which allows us
to find again a direct proof of \eqref{ASCVGDN} and \eqref{ANDN} and to propose new standard results for the sequence
$(D_n)$ such as a law of iterated logarithm, a quadratic strong law and a functional central limit theorem.
The main results of the paper are given in Section \ref{S-MR}. We establish a SLDP for the sequence $(D_n)$
as well as a sharp concentration inequality involving the rate function $I$.
Three keystone lemmas are analyzed in Section \ref{S-LEM}. All technical proofs are postponed to Sections \ref{S-SLDP} to \ref{S-SR}.


\section{Our martingale approach}
\label{S-MG}

 {We start by describing  precisely the  construction of  the sequence $(D_n)$ on a unique probability space. Let us remark that this construction can be naturally linked to a generalized P\'olya urns, see Remark \ref{REM-Urn} below. We consider a sequence $(V_n)$ of independent random variables uniformly distributed on $\{1,...,n\}$. Then, we set $\pi_1=(1)$ and, for each $n\geq 1$, 
we define recursively the permutation $\pi_{n+1}$ as
\begin{equation}\label{DEFPiN}
\pi_{n+1}(k) = \left\{ \begin{array} {ccc}
\pi_n(k) & \textrm { if }& k<V_{n+1},\\ 
n+1 & \textrm { if }& k=V_{n+1},\\ 
\pi_n(k-1) & \textrm { if }& k>V_{n+1}.\end{array}\right.
\end{equation}
By a direct recursive argument, it is clear that for each $n\geq 1$,  $\pi_n$ is uniformly distributed on $\cS_n$.}
  {Moreover, as explained in  \cite{Ozdemir2021}, it follows from \eqref{DEFDN} and \eqref{DEFPiN} that for all $n \geq 1$,}
\begin{equation}
\label{matrice-transition}
   \dP(D_{n+1}  = D_{n}+d|   {\mathcal{F}_n}) =\left \{ \begin{array}{ccc}
    {\displaystyle \frac{n-D_n}{n+1} }  & \text{ if } & d=1, \vspace{1ex}\\
     {\displaystyle \frac{D_n+1 }{n+1} } & \text{ if }  & d=0, 
   \end{array} \nonumber \right.
\end{equation}
   {with $\cF_n=\sigma(D_1, \ldots,D_n)$.}
This    {means} that 
\begin{equation}
\label{DECDN}
D_{n+1} = D_n + \xi_{n+1}
\end{equation}
where the conditional distribution of $\xi_{n+1}$ given $\cF_n$ is the Bernoulli
$\cB(p_n)$ distribution with parameter
\begin{equation*}
p_n=\frac{n-D_n}{n+1}.
\end{equation*} 
Since $\dE[\xi_{n+1} | \cF_n]=p_n$ and $\dE[\xi_{n+1}^2 | \cF_n]=p_n$, we deduce from \eqref{DECDN} that
\begin{eqnarray}
\label{CEDN}
\dE[D_{n+1} | \cF_n] & =& \dE[D_n+ \xi_{n+1} | \cF_n]=
D_n + p_n \hspace{1cm} \text{a.s.} \\
\label{CVDN}
\dE[D_{n+1}^2 | \cF_n] &=& \dE[(D_n+ \xi_{n+1})^2| \cF_n]=D_n^2+ 2p_nD_n +p_n
\hspace{1cm} \text{a.s.} 
\end{eqnarray}
Moreover, let $(M_n)$ be the sequence defined for all $n \geq 1$ by
\begin{equation}
\label{DEFMN}
M_{n} = n \Bigl( D_n - \frac{n-1}{2} \Bigr).
\end{equation}
We obtain from \eqref{CEDN} that
\begin{eqnarray*}
\dE[M_{n+1} | \cF_n] &=& (n+1) \Bigl( D_n +p_n - \frac{n}{2} \Bigr)=
(n+1) \Bigl( \frac{n}{n+1}D_n - \frac{n(n-1)}{2(n+1)} \Bigr), \\
&=&n \Bigl( D_n - \frac{n-1}{2} \Bigr)=M_n \hspace{1cm} \text{a.s.} 
\end{eqnarray*}
which means that $(M_n)$ is a locally square integrable martingale. We deduce from
\eqref{CVDN} that its predictable quadratic variation is given by 
\begin{equation}
\label{PQVMN}
\langle M \rangle_n  = \sum_{k=1}^{n-1} \dE[(M_{k+1}-M_k)^2|\cF_{k}]=
 \sum_{k=1}^{n-1} (k-D_k)(D_k+1) \hspace{1cm} \text{a.s.} 
\end{equation}
The martingale decomposition \eqref{DEFMN} allows us to find again all the asymptotic results
previously established for the sequence $(D_n)$ such as the almost sure convergence \eqref{ASCVGDN} and 
the asymptotic normality \eqref{ANDN}.
Some improvements to these standard results are as follows. To the best of our knowledge, the quadratic strong law and the law of iterated logarithm are new.

\begin{prop}
\label{P-QSLLIL}
We have the quadratic strong law
\begin{equation}
\label{QSL}
\lim_{n\rightarrow \infty} 
\frac{1}{\log n} \sum_{k=1}^n  \Bigl(\frac{D_k}{k} - \frac{1}{2}\Bigr)^2  =\frac{1}{12}
\hspace{1cm} \text{a.s.}
\end{equation}
Moreover, we also have the law of iterated logarithm 
\begin{eqnarray}
\limsup_{n \rightarrow \infty}  \Bigl(\frac{n}{2 \log \log n}\Bigr)^{1/2} \Bigl(\frac{D_n}{n}-\frac{1}{2}\Bigr) 
&=& -\liminf_{n \rightarrow \infty}  \Bigl(\frac{n}{2 \log \log n}\Bigr)^{1/2} \Bigl(\frac{D_n}{n}-\frac{1}{2}\Bigr) 
\notag\\
&=& \frac{1}{\sqrt{12}} \hspace{1cm}\text{a.s.}
\label{LIL}
\end{eqnarray}
In particular, 
\begin{equation}
\limsup_{n \rightarrow \infty} \Bigl(\frac{n}{2 \log \log n} \Bigr)
\Bigl(\frac{D_n}{n} - \frac{1}{2}\Bigr)^2 =\frac{1}{12}
\hspace{1cm} \text{a.s.}
\label{LILSUP}
\end{equation}
\end{prop}

\noindent
Denote by $D([0,\infty[)$ the Skorokhod space of right-continuous functions
with left-hand limits.  {The functional central limit theorem extends the asymptotic normality \eqref{ANDN}, see a similar result in \cite{Gouet1993} using
generalized P\'olya's urns.}

\begin{prop}
\label{P-FCLT}
We have the distributional convergence in $D([0,\infty[)$,
\begin{equation}
\label{FCLT}
\Big( \sqrt{n}\Big(\frac{D_{\lfloor n t \rfloor}}{\lfloor n t \rfloor}-\frac{1}{2}\Big), t \geq 0\Big) \Longrightarrow \big( W_t, t \geq 0 \big)
\end{equation}
where $(W_t)$ is a real-valued centered Gaussian process starting at the origin with covariance given, for all $0<s \leq t$, by
$$
\dE[W_s W_t]= \frac{s}{12 t^2}.
$$
In particular, we find again the asymptotic normality \eqref{ANDN}.
\end{prop}

\begin{proof}
The proofs are postponed to Section \ref{S-SR}
\end{proof}

{ 
\begin{rem}
\label{REM-Urn}
Relation \eqref{matrice-transition} allows to see the sequence $(D_n)$ as the sequence of the number of white balls in a two colors generalized P\'olya urn \cite{Friedman1949} with the following rule: at each step, one ball is drawn at random and then replaced with an additional ball of the opposite color.
\end{rem}
}


\section{Main results}
\label{S-MR}

\subsection{Sharp large deviations and concentration}
\noindent
Our first result concerns the SLDP for the sequence $(D_n)$ which nicely extends the LDP previously established by Bryc, Minda and Sethuraman \cite{Bryc2009}. For any positive real number $x$, denote $\{x\}=\lceil x \rceil - x$.
\begin{thm}
\label{T-SLDP}
For any $x$ in $]1/2,1[$, we have on the right side
\begin{equation}
\label{SLDPR}
\dP \Big(\frac{D_n}{n} \geq x \Big) = \frac{\exp(-nI(x)-\{nx\}t_x)}{\sigma_x t_x \sqrt{2\pi n}} 
\big[ 1+o(1) \big]
\end{equation} 
where the value $t_x$ is the unique solution of $L^\prime(t_x)=x$ and $\sigma_x^2=L^{\prime \prime}(t_x)$.
\end{thm}
\noindent
Our second result is devoted to an optimal concentration inequality involving the rate function $I$.

\begin{thm}
\label{T-concentration}
For any $x \in ]1/2,1[$ and for all $n \geq 1$, we have the concentration inequality
\begin{equation}
\label{CID}
    \dP\Big(\frac{D_n}{n} \geq x \Big) \leq  P(x)\frac{\exp(-n I(x)-\{nx\}t_x)}{\sigma_x t_x \sqrt{2\pi n}}
\end{equation}
where the prefactor can be taken as
\begin{align*}
    P(x) =  \sqrt{\frac{t_x^2 + \pi^2}{t_x^2}}+ \left(1+\frac1\pi+\frac{ 2\sqrt{t_x^2+\pi^2} }{\pi^2 -4} \right) \sqrt{\frac{\pi^2(t_x^2+4)}{4}}  .
\end{align*}
\end{thm}

\begin{rem}
Let us denote $A_n = A_n(\pi_{ n})$ the random variable counting the number of ascents of a permutation $\pi_{ n} \in S_n$. Then, it is clear that 
$D_n(\pi_{ n})+A_n(\pi_{ n})=n-1$. Moreover, by a symmetry argument, $D_n$ and $A_n$ share the same distribution. In particular, $D_n$ has the same distribution as $(n-1) - D_n$. Consequently, for all $x \in ]1/2,1[$, we have
\[ \dP \Big(\frac{D_n+1}{n} \leq 1-x \Big) = \dP \Big(\frac{D_n}{n} \geq x \Big)  \]
which allows to extend immediately the previous results to the left side.
\end{rem}

 {
\begin{rem}
One can observe from \eqref{SLDPR} or \eqref{CID} that for all $\varepsilon >0$,
$$
\sum_{n=1}^\infty \dP \Big( \Big| \frac{D_n}{n} -\frac{1}{2} \Big|>\veps \Big) <+\infty. 
$$
That 
is the complete convergence of $(D_n/n)$ to $1/2$, which directly implies the almost sure convergence \eqref{ASCVGDN} for any construction of the sequence $(D_n)$.
\end{rem}
}

\subsection{A more direct approach}
\label{sub-Tanny}
An alternative approach to prove SLDP and concentration inequalities for the sequence $(D_n)$ relies on a famous result of Tanny \cite{Tanny1973}
which says that the distribution of $D_n$ is nothing else than the one of the integer part of the sum $S_n$ of independent and identically distributed random variables. More precisely, let $(U_n)$ be a sequence of independent random variables
sharing the same uniform distribution on $[0,1]$. Denote
$$
S_n= \sum_{k=1}^n U_k.
$$
Then, we have from \cite{Tanny1973} 
that for all $k \in [\![ 0, n-1 ]\!]$,
\begin{equation}
    \label{Tanny}
    \dP(D_n=k)=\dP(\lfloor S_n \rfloor =k)=\dP(k \leq S_n <k+1).
\end{equation}

\noindent
It simply means that the distribution of $D_n$ is the one of the integer part of the the Irwin-Hall distribution.
Identity \eqref{Tanny} is somewhat miraculous and it is really powerful in order to carry out a sharp analysis of the sequence $(D_n)$. Once again, we would like to emphasize that this direct approach is only relevant for the study of $(D_n)$ while our martingale approach 
is much more general. A direct proof of Theorem \ref{T-SLDP} is provided in Section \ref{S-PRU} relying on identity
\eqref{Tanny}. It is also possible to use this direct approach in order to establish a sharp concentration inequality with the same shape as \eqref{T-concentration}, see Remark \ref{rem-concentration-direct} below.


\subsection{Further considerations on concentration inequalities}


We wish to compare our concentration inequality \eqref{CID} with some classical ones. The first one is given by 
the well-known Azuma-Hoeffding inequality \cite{Bercu2015}.
It follows from \eqref{PQVMN} that the predictable quadratic variation $\langle M \rangle_n$ of the martingale $(M_n)$ satisfies
$$ 
\langle M \rangle_n \leq \frac{s_n}{4}
\hspace{1cm} \text{where} \hspace{1cm}
s_n=\sum_{k=2}^n k^2.
\vspace{-1ex}
$$
In addition, its total quadratic variation reduces to
$$
[M]_n=\sum_{k=1}^{n-1}(M_{k+1}-M_{k})^2=s_n.
$$
Consequently, we deduce from an improvement of Azuma-Hoeffding inequality given by inequality (3.20) in \cite{Bercu2015} that
for any $x \in ]1/2,1[$ and for all $n \geq 1$, 
\begin{equation}
\label{Azuma-Hoeeding}
    \dP\Big(\frac{D_n}{n} \geq x \Big) \leq \exp\Big(\!\! - \frac{2n^4}{s_n}\Big( x - \frac{1}{2} \Big)^2 \Big).
\end{equation}
One can observe that \eqref{CID} is much sharper than \eqref{Azuma-Hoeeding} for all values of $x \in ]1/2,1[$.
Furthermore, by using \eqref{Tanny}, we can also infer a concentration inequality by means of Chernoff's inequality. Indeed, for any $x \in ]1/2,1[$ and for all $n \geq 1$, we have
\begin{eqnarray}
    \nonumber
    \dP\Big(\frac{D_n}{n} \geq x \Big) &=& \dP\Big(\sum_{k=1}^n U_k \geq \lceil  n x \rceil\Big) \leq \exp \big( nL(t_x)-t_x \lceil nx \rceil \big) \\
    &\leq& \exp\big(-n I(x)-\{nx\}t_x\big)
    \label{Chernoff}
\end{eqnarray}
which is also rougher than \eqref{CID}.

\section{Three keystone lemmas}
\label{S-LEM}


Denote by $m_n$ the Laplace transform of $D_n$ defined, for all $t \in \dR$, by 
\begin{equation}
\label{Def-Laplace}
m_n(t)=\dE[\exp(t D_n)].
\end{equation}
One can observe that $m_n(t)$ is finite for all $t \in \dR$ and all $n \geq 1$ since $D_n$ is finite.
Let us introduce the generating function defined, for all $t \in \dR$
and for all $z \in \dC$, by
$$
F(t,z)=\sum_{n=0}^\infty m_n(t) z^n
$$
where the initial value is such that, for all $t\in \dR$, $m_0(t)=1$. Let us notice that the radius of convergence, denoted $R^F(t)$, should depend on $t$ and is positive since $|m_n(t)| \leq e^{n|t|}$. Moreover, we easily have for all
$|z|<R^F(0)=1$,
$$F(0,z) = \frac{1}{1-z}.$$
Our first lemma is devoted to the calculation of 
the generating function $F$,   {see also \cite{Bryc2009} page 865, where a similar expression was given without proof. One can observe that $k_0$ should be replaced by $1-k_0$. Let us also remark that the recursive equation \eqref{RECEQ2} was already given in \cite{Friedman1949}, Section 4.}


\begin{lem}
\label{L-GeneFunction}
For all $t\in \dR$, we have
\begin{equation}
\label{Rt}
R^F(t) =  \frac{t}{e^t-1}.
\end{equation}
Moreover, for all $t \in \dR$ and for all $z \in \dC$ such that $|z| < R^F(t)$,
\begin{equation}
\label{SOLF}
F(t,z)= \frac{1-\exp(-t)}{1-\exp((e^t-1)z-t)}.
\end{equation}
\end{lem}


\begin{proof}
It follows from \eqref{DECDN} that for all $t\in \dR$ and for all $n \geq 1$,
\begin{eqnarray}
m_{n+1} (t) & = & \dE\big[\exp(t D_{n+1})\big]=\dE\big[\exp(t D_n)\dE[\exp(t\xi_{n+1})|\cF_n]\big], \notag \\
& = & \dE\big[\exp(t D_n) p_n e^t+\exp(t D_n)(1-p_n)\big] , \notag \\
& = &  m_n(t)+ (e^t -1) \dE\big[p_n\exp(t D_n)\big].  
\label{RECEQ1}
\end{eqnarray}
However, we already saw that
\begin{equation*}
p_n= \frac{n-D_n}{n+1},
\end{equation*}
which implies that 
$$
\dE\big[p_n\exp(t D_n)\big]=\frac{n}{n+1}m_n(t)-\frac{1}{n+1}m_n^\prime(t).
$$
Consequently, we obtain from \eqref{RECEQ1} that for all $t \in \dR$ and for all $n\geq 1$,
\begin{equation}
\label{RECEQ2}
m_{n+1}(t)=\Big(\frac{1+n e^t}{n+1}\Big)m_n(t)+\Big(\frac{1-e^t}{n+1}\Big)m_n^\prime(t).
\end{equation}
One can observe that \eqref{RECEQ2} remains true for $n=0$.
We deduce from \eqref{RECEQ2} that for all $|z|<R^F(t)$,
\begin{eqnarray*}
\frac{\partial F(t,z)}{\partial z} & = & \sum_{n=1}^\infty n m_n(t) z^{n-1}=\sum_{n=0}^\infty (n+1) m_{n+1}(t) z^{n}, \\
& = & \sum_{n=0}^\infty (1+n e^t)m_n(t) z^n + \sum_{n=0}^\infty (1-e^t)m_n^\prime(t)z^n,\\
& = &  F(t,z)+e^t z\frac{\partial F(t,z)}{\partial z} + (1- e^t) \frac{\partial F(t,z)}{\partial t}, 
\end{eqnarray*}
where the last equality comes from the fact that $|m_n^\prime(t)| \leq n m_n(t)$ allows us to apply dominated convergence theorem in order to differentiate in $t$ the series.
Hence, we have shown that the generating function $F$ is the solution of the following partial differential equation 
\begin{equation}
\label{DIFEQF}
(1- e^t z)\frac{\partial F(t,z)}{\partial z} + (e^t-1) \frac{\partial F(t,z)}{\partial t} 
=   F(t,z)
\end{equation}
with initial value
\begin{equation}
\label{INITVAL} 
F(t,0)=m_0(t)=1.
\end{equation}
We shall now proceed as in \cite{Flajolet2005} in
order to solve the partial differential equation \eqref{DIFEQF}
via the classical method of characteristics, see e.g. \cite{Zwillinger1989}. Following this method, one first associates to the linear first-order partial differential equation \eqref{DIFEQF}, the ordinary differential system given by
\begin{equation}
    \label{ODIFSYST}
    \frac{dz}{1- e^t z} = \frac{dt}{e^t-1} = \frac{dw}{w}
\end{equation}
where $w$ stands for the generating function $F$. 
We assume in the sequel that $t>0$, inasmuch as the
proof for $t<0$ follows exactly the same lines. The equation binding $w$ and $t$ can be easily solved and we obtain that 
\begin{equation}
\label{EQC1}
w = C_1 (1-e^{-t}).
\end{equation}
The equation binding $z$ and $t$ leads to the ordinary differential equation
$$\frac{dz}{dt} = -\frac{e^t}{e^t-1} z + \frac{1}{e^t-1}.$$
We find by the variation of constant method that 
\begin{equation}
\label{EQC2}
(e^{t}-1)z-t =C_2.
\end{equation}
According to the method of characteristics, the general solution of \eqref{DIFEQF} is obtained by coupling \eqref{EQC1} and \eqref{EQC2}, namely 
\begin{equation}
\label{EQPSI}
 C_1 = f(C_2) 
\end{equation}
 {where $f$ is a function which can be explicitly calculated from the boundary value in \eqref{INITVAL}}. We deduce from the conjunction of
\eqref{EQC1}, \eqref{EQC2} and \eqref{EQPSI} that
for all $t>0$ and for all $z \in \dC$ such that $|z|<R^F(t)$,
\begin{equation}
\label{EQFPSI}
F(t,z)=(1-e^{-t})f\left((e^{t}-1)z-t\right).
\end{equation}
It only remains to determine the exact value of the function $f$ by taking into account the initial condition \eqref{INITVAL}.
We obtain from \eqref{EQFPSI} with $z=0$ and replacing $-t$ by $t$ that
\begin{equation}
\label{EPSI}
f(t) = \frac{1}{1-e^{t}}.
\end{equation}
Finally, the explicit solution \eqref{SOLF} clearly follows from \eqref{EQFPSI} and \eqref{EPSI}. Moreover, one can observe that the radius of convergence comes immediately from \eqref{SOLF}, which completes the proof of Lemma \ref{L-GeneFunction}.
\end{proof}

\noindent
The global expression \eqref{SOLF} of the generating function $F$ allows us to deduce a sharp expansion of the Laplace transform $m_n$ of $D_n$, as follows.


\begin{lem}
\label{LDEVMn}
For any $t\neq 0$, we have
\begin{equation}
\label{MNT}
m_n(t)=\Big(\frac{1-e^{-t}}{t} \Big)\Big(\frac{e^{t}-1}{t}\Big)^n\big(1+r_n(t)\big)
\end{equation}
where the remainder term $r_n(t)$ goes exponentially fast to zero as 
\begin{equation}
\label{UBRemainder}
|r_n(t)|\leq  |t|e \Bigl( 1 + \frac{1}{\pi} + \frac{2+n}{\sqrt{t^2+4\pi^2}}\Bigr) 
\Big( 1+ \frac{4 \pi^2}{t^2} \Big)^{-n/2}.
\end{equation}
\end{lem}


%
%

\begin{proof}
In all the proof, we assume that $t \neq 0$. It follows from  \eqref{SOLF} that $F$ is a meromorphic function on $\dC$ with simple poles given, for all $\ell \in \dZ$, by
\begin{equation}
\label{POLES}
z_\ell^F(t) = \frac{t+2i \ell \pi}{e^t-1}.
\end{equation}
By a slight abuse of notation, we still denote by $F$ this meromorphic extension. 
Hereafter, for the sake of simplicity, we shall consider the function $\cF$ defined,
for all $z \in \dC$, by
\begin{equation}
\label{DEFCF}
\cF(t,z)= \frac{ 1}{1-e^{-t}} F(t,z)=f( \xi(t,z))
\end{equation}
where the function $f$ was previously defined in \eqref{EPSI} and the function $\xi$ is given, for all $z \in \dC$, by
\begin{equation}
\label{DEFXI}
\xi(t,z)= (e^{t}-1)z-t.
\end{equation}
By the same token, we shall also introduce the functions $\cG$ and $\cH$ defined,
for all $z \in \dC$, by
\begin{equation}
\label{DEFCGH}
\cG(t,z)= g( \xi(t,z)) \hspace{1cm} \text{and} \hspace{1cm} \cH(t,z)= h( \xi(t,z))
\end{equation}
where $g$ and $h$ are given, for all $z \in \dC^*$, by
\begin{equation}
\label{gh}  
g(z)= -\frac{1}{z} \hspace{1cm} \text{and} \hspace{1cm} h(z)=  \frac{1}{1-e^z} + \frac{1}{z}.
\end{equation}
One can immediately observe from \eqref{gh} that $\cH = \cF-\cG$ which means that we have subtracted to $\cF$ its simple pole at $0$ to get $\cH$.
Given a function $\cK(t,z)$ analytic in $z$ on some set $\{(t,z)\in \dR\times \dC, |z|\leq R^\cK(t) \}$, we denote by $m_n^\cK(t)$ the coefficient of its Taylor series at point $(t,0)$, that is 
\[
\cK(t,z)= \sum_{n=0}^\infty m_n^\cK(t) z^n.
\]
Thanks to this notation, we clearly have $m_n^{\cF} (t)= m_n^{\cG} (t)+ m_n^{\cH} (t)$. Moreover, we deduce from \eqref{DEFCF} that
\begin{equation}
\label{DECMNF}
m_n^{\cF} (t)= \frac{ 1}{1-e^{-t}} m_n^F(t).
\end{equation}
The first coefficient $m_n^{\cG} (t)$ can be explicitly computed by
\begin{equation}
\label{CALCMNG}
m_n^{\cG} (t)=\frac{1}{t}  \Big(\frac{e^t-1}{t}\Big)^n.
\end{equation}
As a matter of fact, for all $z \in \dC$ such that $|z| < R^{\cG}(t)=t(e^t-1)^{-1}$, it follows from \eqref{DEFXI} and \eqref{DEFCGH} that
$$
\cG(t,z)=-\frac{1}{\xi(t,z)}= \frac{1}{t-(e^{t}-1)z} = 
\frac{1}{t}  \sum_{n= 0}^\infty  \Big(\frac{e^t-1}{t}\Big)^n z^n.
$$
Consequently, as $m_n(t)=m_n^F(t)$, we obtain from \eqref{DECMNF} that
$$
m_n(t) = (1- e^{-t}) \big( m_n^{\cG} (t)+ m_n^{\cH} (t) \big)= (1- e^{-t}) m_n^{\cG}(t) \big(1+r_n(t)\big)
$$
which leads via \eqref{CALCMNG} to
\begin{equation}
\label{CALCMN}
m_n(t)= \Big(\frac{1-e^{-t}}{t} \Big)\Big(\frac{e^{t}-1}{t}\Big)^n\big(1+r_n(t)\big)
\end{equation}
where the remainder term $r_n(t)$ is the ratio $r_n(t)=m_n^{\cH} (t)/m_n^{\cG} (t)$. From now on, we shall focus
our attention on a sharp upper bound for $m_n^{\cH} (t)$. The function $h$ is meromorphic with simple poles at the points $2i\pi \dZ^*$.
Moreover, for a given $t \neq 0$, $z$ is a pole of $\cH$ if and only if $(e^t-1)z - t$ is a pole of $h$. Hence, the poles of $\cH$ are given, for all $\ell \in \dZ^*$, by 
$$
z_\ell^\cH(t)= \frac{t+2i \ell \pi}{e^t-1}.
$$
In addition, its radius of convergence $R^\cH (t)$ is nothing more than the shortest distance between 0 and one of these poles. Consequently, we obtain that
$$
R^\cH(t)= |z_1^\cH(t)|=  \frac{t}{e^t-1} \sqrt{1 + \frac{4\pi^2}{t^2}}.
$$
Furthermore, it follows from Cauchy's inequality that for any $0<\rho(t)<R^\cH(t)$,
\begin{equation}
\label{CAUCHYMN}
|m_n^\cH (t)| \leq  \frac{ \Vert \cH(t,\cdot) \Vert_{\infty, \cC(0,\rho(t))}} {\rho(t)^n }
\end{equation}
where the norm in the numerator stands for
$$ \Vert \cH(t,.)\Vert_{\infty, \cC(0,\rho(t))}= \sup\{ |\cH(t,z)|, |z|= \rho(t)\}.$$
Since $\xi(t,\cC(0,\rho(t)))$ coincides with the circle $\cC( -t, |e^t- 1| \rho(t))$, we deduce from the identity $\cH(t,z)= h( \xi(t,z))$ that 
$$
\Vert \cH(t,\cdot) \Vert_{\infty, \cC(0,\rho(t))} = \Vert h \Vert_{\infty, \cC(-t, |e^t-1|\rho(t))}.
$$
Hereafter, we introduce a radial parameter 
\begin{equation}
\label{CHOICERTA}
\rho(t,\alpha)=\frac{t}{e^t-1} \sqrt{1 + \frac{4 \alpha \pi^2}{t^2}}
\end{equation}
where $\alpha$ is a real number in the interval $]\!-t^2/4\pi^2, 1[$.
We also define the distance between the circle $\cC(-t, |e^t-1|\rho(t, \alpha))$ and the set of the poles of $h$,
\begin{equation*} 
\delta(t,\alpha)= d ( \cC(-t, |e^t-1|\rho(t, \alpha)) , 2i\pi \dZ^*).
\end{equation*}
We clearly have from the Pythagorean theorem that
\begin{equation*} 
   \delta(t,\alpha)= \sqrt{t^2+ 4\pi^2} - \sqrt{t^2+ 4\alpha \pi^2}.
\end{equation*}
In addition, one can easily check that 
$$
\delta(t,\alpha)= \frac{4\pi^2(1-\alpha)}{\sqrt{t^2+ 4\pi^2} + \sqrt{t^2+ 4\alpha \pi^2}}
$$
which ensures that
\begin{equation} \label{DISTPOLEESTIMATE}
   \frac{2\pi^2(1-\alpha)}{\sqrt{t^2+ 4\pi^2} } <  \delta(t,\alpha) < \frac{4\pi^2(1-\alpha)}{\sqrt{t^2+ 4\pi^2} }
\end{equation}
It follows from the maximum principle that
$$
\Vert h \Vert_{\infty, \cC(-t, |e^t-1|\rho(t, \alpha))} \leq \Vert h \Vert_{\infty,\partial \cD (L,\Lambda, \delta(t,\alpha))}
$$
where, for $L>0$ and $\Lambda>0$ large enough, $\cD (L,\Lambda,\delta(t,\alpha))=\cB(L,\Lambda)\cap  {\cA_h(\delta(t,\alpha))}^{ \complement}$ is the domain given by
the intersection of the box
$$
\cB (L,\Lambda)= \{ z\in \dC, |\Re(z)|<L, |\Im(z)|<\Lambda\}
$$
and the complementary set of 
$$
\cA_h (\delta(t,\alpha))= \{ z\in \dC, d(z,2i  \pi \dZ^*) \leq \delta(t,\alpha) \text{ with } |\Im(z)| \geq \pi\}.$$
On the one hand, we have for all $y \in \dR$, $|e^{L+iy}-1|\geq e^{L}-1$ and $|L+iy| \geq L$, implying that for all $y \in \dR$,
$$\Big|h(L+iy) \Big|\leq \frac{1}{e^{L}-1} + \frac{1}{L}.$$
By the same token, we also have for all $y \in \dR$, $|e^{-L+iy}-1|\geq 1- e^{-L}$ and $|-L+iy| \geq L$, leading, for all $y \in \dR$, to
$$\Big|h(-L+iy )\Big|\leq \frac{1}{1- e^{-L}} + \frac{1}{L}.$$
On the other hand, we can choose $\Lambda$ of the form $\Lambda=(2 k+1) \pi $ for a value $k\in \dN^*$ large enough.
Then, we have for all $x \in \dR$, $\exp(x+(2 k+1) i\pi)=-\exp(x)$ and $|x +(2 k+1) i\pi| \geq (2k+1) \pi$, implying that for all $x \in \dR$,
\begin{equation}
\label{ESTh-LIGNE}
\Big|h(x+(2 k+1)i\pi)\Big|\leq 1 + \frac{1}{(2 k+1) \pi}.
\end{equation}
By letting $L$ and $\Lambda$ go to infinity, we obtain that
$$
\Vert h \Vert_{\infty, \cC(-t, (e^t-1)\rho(t, \alpha))} \leq \max\big(1,\Vert h \Vert_{\infty,\partial \cD (\delta(t,\alpha))}\big)
$$
where $\cD (\delta(t,\alpha))$ is the domain 
$
\cD (\delta(t,\alpha))= {\cA_h(\delta(t,\alpha))}^{ \complement}.
$
We clearly have from \eqref{gh} that for all
$z\in \partial \cD (\delta(t,\alpha))$ with $|\Im(z)| > \pi$, 
\begin{equation}
\label{ESTh}
|h(z)|\leq  |f(z)| + 
\frac{1}{|z|}
\leq |f(z)| +
\frac{1}{\pi}.
\end{equation}
Moreover, it follows from tedious but straightforward calculations that
$$
\inf_{z \in \dC, |z|= \delta(t,\alpha)}\big| 1-e^z \big|= 1-e^{-\delta(t,\alpha)},
$$
which ensures that
\begin{equation}\label{ESTf}
|f(z)| \leq\frac{1}{1-e^{-\delta(t, \alpha)}}.
\end{equation}
In addition, we obtain from \eqref{ESTh-LIGNE} that for all $z\in \partial \cD (\delta(t,\alpha))$ with $|\Im(z)|=\pi$,
\begin{equation}
\label{ESTh-LIGNE2}
|h(z)| \leq 1+  \frac{1}{\pi}.
\end{equation}
Hence, we find from  \eqref{ESTh}, \eqref{ESTf} and \eqref{ESTh-LIGNE2} that
\begin{equation*}
\Vert h \Vert_{\infty,\partial \cD (\delta(t,\alpha))} \leq \frac{1}{1-e^{-\delta(t, \alpha)}}+ \frac{1}{\pi}. 
\end{equation*}
We were not able to find an explicit maximum for the previous upper-bound. However, it is not hard to see that
\begin{equation*}
\frac{1}{1-e^{-\delta(t, \alpha)}}\leq 1+\frac{1}{\delta(t, \alpha)},
\end{equation*}
which gives us
\begin{equation}
\label{ESThnotex}
\Vert h \Vert_{\infty,\partial \cD (\delta(t,\alpha))} \leq  1+ \frac{1}{\pi}+\frac{1}{\delta(t, \alpha)}. 
\end{equation}

Consequently,
we deduce from \eqref{CAUCHYMN}, \eqref{CHOICERTA}, \eqref{DISTPOLEESTIMATE} and \eqref{ESThnotex} that for all $t \neq 0$ and for all $n \geq 1$,
\begin{equation}
\label{ESTIMATION:mnH1}
|m_n^\cH(t)| \leq \Big(\frac{e^t -1 }{t}\Big)^n \varphi_n(t,\alpha)  
\end{equation}
where
\begin{equation}
\label{DEFvarphin}
\varphi_n(t,\alpha)= \Big(  1+ \frac{1}{\pi}+ \frac{\sqrt{t^2+ 4\pi^2} }{2\pi^2(1-\alpha)}  \Big) \Big( 1+ \frac{4 \alpha \pi^2}{t^2} \Big)^{-n/2}.
\end{equation}
For the sake of simplicity, let $\Phi$ be the function defined, for all 
$\alpha \in ]\!-t^2/4\pi^2,1[$, by
$$
\Phi(\alpha)= \Big( \frac{1 }{1-\alpha} \Big) \Big( 1+ \frac{4 \alpha \pi^2}{t^2} \Big)^{-n/2}.
$$
One can easily see that $\Phi$ is a convex function reaching its minimum for the value
$$
 \alpha=1-\Big(1 + \frac{t^2}{4\pi^2}\Big)\Big(1+\frac{n}{2}\Big)^{-1}.
$$
 {Some numerical experiments show that this explicit value seems to be not far from being the optimal value} that minimizes $\varphi_n(t,\alpha)$.
By plugging $\alpha$ into \eqref{DEFvarphin}, we obtain from \eqref{ESTIMATION:mnH1} that for all $t \neq 0$ and for all $n \geq 1$,
\begin{eqnarray}
|m_n^\cH(t)| &\leq& \Big(\frac{e^t -1 }{t}\Big)^n
\Bigl( 1 + \frac{1}{\pi} + \frac{2+n}{\sqrt{t^2+4\pi^2}}\Bigr) \Big( 1+ \frac{2}{n} \Big)^{n/2}
\Big( 1+ \frac{4 \pi^2}{t^2} \Big)^{-n/2}, \notag \\
&\leq& e\Big(\frac{e^t -1 }{t}\Big)^n \Bigl( 1 + \frac{1}{\pi} + \frac{2+n}{\sqrt{t^2+4\pi^2}}\Bigr) 
\Big( 1+ \frac{4 \pi^2}{t^2} \Big)^{-n/2}.
\label{ESTIMATION:mnH2}
\end{eqnarray}
Finally, \eqref{UBRemainder} follows from \eqref{CALCMNG} together with \eqref{ESTIMATION:mnH2}, which completes the proof of Lemma \ref{LDEVMn}.

\end{proof}
\noindent
We now focus our attention on a complex estimate of the Laplace transform $m_n$ of $D_n$ since $m_n$ clearly extends to an analytic function on $\dC$. More precisely, our goal is to compute an estimate 
of $m_n(t+iv)$ for $t \neq 0$ and for $v \in \dR$ such that $|v|< \pi$. Note that $m_n$ is $2i\pi$ periodic.

\begin{lem}
\label{LDEVMncompl}
For any $t \neq 0$ and for all $v \in \dR$ such that $|v|<\pi$, we have
\begin{equation}
\label{MNTcompl}
m_n(t+iv)=\Big(\frac{1-e^{-(t+iv)}}{t+iv} \Big)\Big(\frac{e^{t+iv}-1}{t+iv}\Big)^n\big(1+r_n(t+iv)\big)
\end{equation}
where the remainder  term $r_n(t+iv)$ is exponentially negligible and satisfies
\begin{equation}
\label{UBComplexRemainder}
|r_n(t+iv)|\leq   \sqrt{t^2+v^2} \Big( 1 + \frac{1}{\pi} + \frac{\sqrt{t^2+4\pi^2}}{\pi(\pi-|v|)} \Big)\Big( \frac{t^2+v^2}{t^2+\pi^2}\Big)^{n/2}.
\end{equation}
Moreover, for any $t \neq 0$ and for all $v \in \dR$ such that $|v|\leq \pi$, we also have the alternative upper-bound
\begin{eqnarray}
\label{UBComplexRemainder2}
\nonumber
|m_n(t+iv)| &\leq &   |1-e^{-(t+iv)}|\Big(\frac{e^t-1}{t}\Big)^n\Bigg( \frac{1}{\sqrt{t^2+ v^2}} \exp \Big(\!\! -n \frac{t^2 L''(t)}{t^2+\pi^2} \frac{v^2}{2} \Big) \\
& + & 
\Big(1+\frac1\pi+\frac{ 2\sqrt{t^2+\pi^2} }{\pi^2 -4} \Big) \exp \Big(\!\!-n \frac{4t^2 L''(t)}{\pi^2(t^2+4)} \frac{v^2}{2} \Big) \Bigg)
\end{eqnarray}
where the second derivative of the asymptotic cumulant generating function $L$ is the positive funnction given by
\begin{equation}
\label{SECONDDERL}
L''(t) = \frac{(e^t-1)^2 -t^2 e^t}{(t(e^t-1))^2}.
\end{equation}

\end{lem}

\begin{proof}
We still assume in all the sequel that $t \neq 0$. We shall also extend $F(t,z)$ in the complex plane with respect of the first variable, 
$$
F(t+iv,z) = \sum_{n= 0}^\infty m_n(t+iv) z^n
$$
where the initial value is such that $m_0(t+iv)=1$.
Since $|m_n(t+iv)|\leq  m_n(t)$, the radius of convergence in $z$ of $F(t+iv, \cdot )$ is at least the one for $v=0$. 
Moreover, the poles of $F(t+iv,\cdot)$ are given, for all $\ell \in \dZ$, by
\begin{equation}
    \label{POLESV}
    z_\ell^F(t+iv) = \frac{(t+iv)+2i \ell \pi}{e^{(t+iv)}-1}.
\end{equation}
Consequently, for all $v \in \dR$ such that $|v|< \pi$,  
$$R^F (t+iv)= \frac{|t+iv|}{|e^{t+iv}-1|}.$$
As in the proof of Lemma \ref{LDEVMn}, we can split $F(t+iv,z)$ into two terms,
$$
F(t+iv,z)= \big(1-e^{-(t+iv)}\big) \big( \cG(t+iv,z) + \cH(t+iv,z) \big)
$$
where we recall from \eqref{DEFCGH} that for all $z \in \dC$ and for all $v \in \dR$ such that $|v|< \pi$,
$$
\cG(t+iv,z)=  g ( \xi(t+iv,z) ) \hspace{1cm} \text{and} \hspace{1cm}  \cH(t+iv,z)=  h ( \xi(t+iv,z) )
$$
where $g$ and $h$ are given, for all $z \in \dC^*$, by
$$
g(z)= -\frac{1}{z} \hspace{1cm} \text{and} \hspace{1cm} h(z)=  \frac{1}{1-e^z} + \frac{1}{z}
$$
and the function $\xi$ is such that
$$
\xi(t+iv,z)= (e^{t+iv}-1)z-(t+iv).
$$
By holomorphic extension, we deduce from \eqref{CALCMNG} that 
$$
m_n^{\cG} (t+iv)=\frac{1}{t+iv}  \Big(\frac{e^{t+iv}-1}{t+iv}\Big)^n.
$$
Moreover, the poles of $\cH(t+iv)$ are given, for all $\ell \in \dZ^*$, by 
$$
z_\ell^\cH(t+iv)= \frac{t + i(v+2 \ell \pi) }{e^{t+iv}-1}.
$$
Hence, we obtain that for all $v \in \dR$ such that $|v|< \pi$,
$$
R^\cH(t+iv) = \frac{ \sqrt {t^2 + (2  \pi -|v|) ^2} }{|e^{t+iv}-1|} > \frac{ \sqrt {t^2 + v ^2} }{|e^{t+iv}-1|} = R^F(t+iv).
$$
It follows once again from Cauchy's inequality that for any $0<\rho(t+iv)<R^\cH(t+iv)$,
\begin{equation}
\label{CAUCHYMNC}
|m_n^\cH (t+iv)| \leq  \frac{ \Vert \cH(t+iv,\cdot) \Vert_{\infty, \cC(0,\rho(t+iv))}} {\rho(t+iv)^n }
\end{equation}
where the norm in the numerator stands for
$$ \Vert \cH(t+iv,.)\Vert_{\infty, \cC(0,\rho(t+iv))}= \sup\{ |\cH(t+iv,z)|, |z|= \rho(t+iv)\}.$$
Since the image of the circle $ \cC(0,\rho(t+iv))$ by the application $\xi(t+iv,\cdot)$ coincides with the circle $\cC(- (t+iv), |e^{t+iv}-1| \rho(t+iv))$, we obtain from 
$\cH(t+iv,z)=  h ( \xi(t+iv,z) )$ that
$$
\Vert \cH(t+iv,\cdot) \Vert_{\infty, \cC(-t,\rho(t+iv))} = \Vert h \Vert_{\infty, \cC(-(t+iv), |e^{t+iv}-1| \rho(t+iv))}.
$$
Hereafter, since $|v|< \pi$, one can take the radius
\begin{equation}\label{CHOICErho_t}
 \rho(t+iv) = \frac{\sqrt{t^2 + \pi^2}}{|e^{t+iv}-1|}.
\end{equation}
Moreover, as in the proof of Lemma \ref{LDEVMn}, denote by $\delta(t+iv)$ the distance between the circle $\cC(-(t+iv), |e^{t+iv}-1| \rho(t+iv)))$ and the set of the poles of $h$,
\begin{equation*} 
\delta(t+iv)= d ( \cC(-(t+iv), |e^{t+iv}-1|\rho(t+iv)) , 2i\pi \dZ^*)
\end{equation*}
It follows from \eqref{CHOICErho_t} and the Pythagorean theorem that
\begin{equation} \label{DISTPOLECALCULComp}
   \delta(t+iv)= \sqrt{t^2+ (2\pi-|v|)^2} - \sqrt{t^2+  \pi^2}. 
\end{equation}
One can observe that  
$$
\delta(t+iv)=\frac{(3\pi -|v|)(\pi -|v|)}{\sqrt{t^2+ (2\pi-|v|)^2} + \sqrt{t^2+  \pi^2}}
$$
which leads to
\begin{equation}
\label{DISTPOLEESTIMATEC}
 \frac{\pi(\pi-|v|)}{\sqrt{t^2+4\pi^2}}< \delta(t+iv)< \pi-| v |.
\end{equation}
 {Using \eqref{CAUCHYMNC} together with \eqref{ESThnotex} and \eqref{DISTPOLEESTIMATEC}, we obtain that}
$$
|m_n^\cH (t+iv)| \leq  \!\Big( 1 + \frac{1}{\pi} + \frac{1}{\delta(t+iv)} \Big) \frac{1}{\rho(t+iv)^n} \leq \!\Big( 1 + \frac{1}{\pi} + \frac{\sqrt{t^2+4\pi^2}}{\pi(\pi-|v|)} \Big) \frac{1}{\rho(t+iv)^n}.
$$
Hence, we find that
\begin{align*}
    m_n(t+iv)& =(1-e^{-(t+iv)})m_n^{\mathcal{G}}(t+iv)(1+ r_n(t+iv)) 
\end{align*}
where the remainder term $r_n(t)$ is the ratio
\begin{equation*}
    r_n(t+iv) = \frac{m_n^{\cH}(t+iv)}{m_n^{\cG}(t+iv)}
\end{equation*}
that satisfies
\begin{equation}
\label{PRCRemainder}
| r_n(t+iv)| \leq  \sqrt{t^2+v^2} \Big( 1 + \frac{1}{\pi} + \frac{\sqrt{t^2+4\pi^2}}{\pi(\pi-|v|)} \Big)\Big( \frac{t^2+v^2}{t^2+\pi^2}\Big)^{n/2}.
\end{equation}
Hereafter, we go further in the analyses of $m_n(t+iv)$ by providing a different upper bound for $m_n^\cH(t+iv)$. Our motivation is that factor 
$$\frac{1}{\pi -|v|}$$ 
in \eqref{PRCRemainder} becomes very large
when $|v|$ is close to $\pi$. Our strategy is not to obtain the best exponent by taking the largest radius, close to the radius of convergence. Instead, we shall consider a smaller radius in order to stay away from the poles, but  not to small in order to still have an exponential term with respect to $m_n^\cF(t)$.
Let $\beta$ be the function defined, or all $|v| < \pi$, by
$$
\beta (v) = \frac{2(1- \cos(v))}{v^2}.
$$
It is clear that $\beta$ is even function, increasing on $[-\pi,0]$ and  decreasing on $[0,\pi]$ with a maximum value $\beta(0)=1$ and such that
$\beta(\pi)=4/\pi^2$. We shall replace the radius, previously given by \eqref{CHOICErho_t}, by the new radius
\begin{equation}\label{CHOICErho_t2}
 \rho(t+iv) = \frac{\sqrt{t^2 +\beta(v) v^2}}{|e^{t+iv}-1|}.
\end{equation}
On can observe that we only replace $\pi^2$ by $\beta(v) v^2=2(1-\cos(v))$. As before,
denote by $\delta(t+iv)$ the distance between the circle $\cC(-(t+iv), |e^{t+iv}-1| \rho(t+iv)))$ and the set of the poles of $h$,
\begin{equation} \label{DISTPOLECALCULComp2}
  \delta(t+iv)= \sqrt{t^2+ (2\pi -|v|)^2} - \sqrt{t^2+  \beta(v) v^2 }. 
\end{equation}
As in the proof of \eqref{DISTPOLEESTIMATEC}, we obtain that
\begin{equation}
\label{DISTPOLEESTIMATEC2}
 \frac{\pi^2 -4 } {2\sqrt{t^2+\pi^2}}< \delta(t+iv)< 2\pi-\Big(1+ \frac{2}{\pi}\Big)| v |
\end{equation}
which ensures that
\begin{eqnarray*}
|m_n^\cH (t+iv)| &\leq & \Big( 1 + \frac{1}{\pi} + \frac{1}{\delta(t+iv)} \Big) \frac{1}{\rho(t+iv)^n}, \\
&\leq & \Big( 1 + \frac{1}{\pi} + \frac{2\sqrt{t^2+\pi^2}}{\pi^2-4} \Big) \Big( \frac{|e^{t+iv}-1|^2} {t^2 + \beta(v) v^2}\Big)^{n/2}.
\end{eqnarray*}
Hereafter, it follows from straightforward calculation that
\begin{eqnarray*}
     \frac{|e^{t+iv}-1|^2} {t^2 + \beta(v) v^2} &=& \frac{(e^t-1)^2 + 2 e^t (1-\cos(v))  }{t^2+ \beta(v) v^2}=\frac{(e^t-1)^2 +  e^t  \beta(v) v^2  }{t^2+ \beta(v) v^2},\\
     &=& \frac{(e^t-1)^2}{t^2}   \left(\frac{t^2}{t^2+ \beta(v) v^2} + \frac{t^2 e^t \beta(v) v^2}{ (e^t-1)^2 (t^2+\beta(v) v^2) } \right),\\
   &=&  \frac{(e^t-1)^2}{t^2} \left(  1 - \Big(\frac{(e^t-1)^2 - t^2e^t }{(e^t-1)^2} \Big)  \frac{\beta(v) v^2}{t^2+\beta(v) v^2} \right).
\end{eqnarray*}
Moreover, we also have from \eqref{DEFL} that for all $t \neq 0$,
\begin{equation}
\label{DER2L}
L''(t) = \frac{(e^t-1)^2 -t^2 e^t}{(t(e^t-1))^2}.
\end{equation}
In addition, by using that for $|v| \leq \pi$, we have 
$$\frac{\beta(v)v^2}{t^2+\beta(v) v^2} \geq \frac{4v^2}{\pi^2(t^2+4)}.$$
Therefore, we deduce from the elementary inequality $1-x \leq \exp(-x)$ that
\begin{equation}
\label{ESTIMATION:mnH3}
    |m_n^\cH(t+iv)| \leq  \Big(1+\frac1\pi+\frac{ 2\sqrt{t^2+\pi^2} }{\pi^2 -4} \Big) \Big(\frac{e^t-1}{t}\Big)^n \exp \Big(\!\!-n \frac{4 t^2L''(t)}{\pi^2(t^2+4)} \frac{v^2}{2} \Big).
\end{equation}
We also recall that 
$$
m_n^{\cG} (t+iv)=\frac{1}{t+iv}  \Big(\frac{e^{t+iv}-1}{t+iv}\Big)^n.
$$
We also have from straightforward calculation that
\begin{eqnarray}
\frac{|e^{t+iv}-1|^2}{|t+iv|^2}&=& \frac{(e^t-1)^2 +  2e^t(1-\cos(v)) }{t^2+  v^2}=\frac{(e^t-1)^2 +  e^t \beta(v) v^2 }{t^2+  v^2} \notag\\
  &=& \frac{(e^t-1)^2}{t^2}   \left(\frac{t^2}{t^2+  v^2} + \frac{t^2 e^t \beta(v) v^2}{(e^t-1)^2 (t^2+ v^2) } \right) \label{MAJORATION-v-grand}\\
  &=&\frac{(e^t-1)^2}{t^2} \left(  1 - \left(\frac{(e^t-1)^2  - t^2e^t\beta(v)}{(e^t-1)^2}\right) \frac{v^2}{t^2+ v^2 } \right) \notag\\
  &=&\frac{(e^t-1)^2}{t^2} \left(  1 - \left(\frac{(e^t-1)^2  - t^2e^t}{(e^t-1)^2} \right) \frac{ v^2 }{t^2+ v^2 } - \frac{ t^2e^t}{(e^t-1)^2} \frac{ (1-\beta(v)) v^2 }{t^2+ v^2 } \right) \notag\\
 &\leq& \frac{(e^t-1)^2}{t^2} \left(  1 - \left(\frac{(e^t-1)^2  - t^2e^t}{(e^t-1)^2} \right) \frac{ v^2 }{t^2+ \pi^2 }  \right)
 \label{MAJORATION}
\end{eqnarray}
since $\beta(v) \leq 1$. Hence, we obtain from \eqref{DER2L} that
\begin{equation}
\label{ESTIMATION:mnG3}
    |m_n^\cG(t+iv)| \leq  \frac{1}{\sqrt{t^2+ v^2}} \Big(\frac{e^t-1}{t}\Big)^n \exp \Big(\!\!-n \frac{t^2 L''(t)}{t^2+\pi^2} \frac{v^2}{2} \Big).
\end{equation}
Finally, we already saw that
$$ 
m_n(t+i v) = (1-e^{-(t+iv)}) (m_n^{\mathcal{G}}(t+iv)+ m_n^{\mathcal{H}}(t+iv)).
$$
Consequently, \eqref{ESTIMATION:mnH3} together with \eqref{ESTIMATION:mnG3} clearly lead to \eqref{UBComplexRemainder2}, which achieves the proof of Lemma \ref{LDEVMncompl}.
\end{proof}





\section{Proof of the sharp large deviation principle.}
\label{S-SLDP}



Let us start by an elementary lemma concerning the asymptotic cumulant generating function $L$ defined by \eqref{DEFL}.

\begin{lem}\label{LEM-pte-L}
The function $L :\dR \to \dR$ is twice differentiable and strictly convex function and its first derivative $L':\dR \to ]0,1[$  is a bijection. In particular, for each $x\in ]0,1[$, there exists a unique value $t_x\in \dR$ such that 
\begin{equation}
\label{RATEFI}
I(x)= x t_x - L(t_x)
\end{equation}
where $I$ is the Fenchel-Legendre transform of $L$. The value $t_x$ is also
characterized by the relation $L'(t_x) =x$ where, for all $t \neq 0$,
\begin{equation}
\label{LFIRSTDER}
 L'(t) = \frac{e^t(t-1)+1}{t(e^t-1)}. 
\end{equation}
Moreover, for all $x \in ]1/2,1[$, $t_x >0$ while for all $x \in ]0, 1/2[$, $t_x<0$.
In addition, for all $t \in \dR$, $L''(t)>0$ as the second derivative of $L$ is given, for all $t \neq 0$, by
\begin{equation}
\label{LSECONDDER}
L''(t) = \frac{(e^t-1)^2 -t^2 e^t}{(t(e^t-1))^2}.
\end{equation}
Finally, the function $L$ can be extended to a  function $L:\dC \setminus 2i \pi \dZ^* \to \dC$ satisfying for all
$v \in \dR$ such that $|v|\leq \pi$, 
\begin{equation}
\label{COMPLEXL}
\Re L(t+iv) \leq  L(t) - C(t) \frac{v^2}{2}
\end{equation}
where 
$$ 
C(t)= \frac{t^2}{t^2+\pi^2} L''(t).
$$
\end{lem}

\begin{proof} We already saw in the previous Section that the calculation of the first two derivatives \eqref{LFIRSTDER} and \eqref{LSECONDDER} of $L$
follows from straightforward calculation. Let us remark that 
$$ \lim_{t \to 0} L'(t)=\frac{1}{2} \hspace{1cm} \text{and}  \hspace{1cm} \lim_{t \to 0} L''(t)=\frac{1}{12},$$
which means that $L$ can be extended as a $C^2$ function on $\dR$.
The above computation also gives that 
$$\lim_{t\to-\infty} L'(t)=0\hspace{1cm} \text{and}  \hspace{1cm} \lim_{t\to+\infty} L'(t)=1.$$ 
We now focus our attention on the complex extension of $L$. We deduce from \eqref{MAJORATION} and \eqref{LSECONDDER} together with the elementary
inequality $\ln(1-x) \leq -x$ that for all $t\neq 0$ and $|v|\leq \pi$,
\begin{eqnarray*}
\Re (L(t+iv))&=& \ln \Big(  \frac{|e^{t+i v} -1 |}{|t+iv|}\Big),\\ 
&\leq& L(t) +  \frac{1}{2} \ln \Big(  1 - t^2 L''(t)  \frac{v^2}{t^2+\pi^2} \Big), \\
&\leq& L(t)-   C(t) \frac{v^2}{2},
\end{eqnarray*}
which completes the proof of Lemma \ref{LEM-pte-L}.
\end{proof}

\noindent
We carry out with an elementary lemma which can be seen as a slight extension of the usual Laplace method.

\begin{lem}
\label{LEMLaplace++}
Let us consider two real numbers $a< 0 <b$, and two functions $f:[a,b]\rightarrow \dC$ and
$\varphi:[a,b] \rightarrow \dC$ such that, for all $\lambda$ large enough,
$$\int_{a}^{b} e^{-\lambda \Re \varphi (u)} |f(u)| du < +\infty.$$
Assume that $f$ is a continuous function in $0$, $f(0)\neq 0$, $\varphi$ is a $C^2$ function in $0$, $\varphi'(0)=0$, $\varphi''(0)$ is a real positive number and there exists a constant $C>0$ such that $\Re \varphi(u) \geq \Re \varphi(0) + Cu^2$. Then, 
we have
\begin{equation}
\label{LAPLACE}
\lim_{\lambda \rightarrow \infty} \sqrt{\lambda} e^{\lambda\varphi(0)}\int_{a}^{b} e^{-\lambda \varphi (u)} f(u) du 
=\sqrt{2\pi}\frac{f(0)}{\sqrt{\varphi''(0)}}.
\end{equation}
\end{lem}

\begin{proof}
First of all, we can assume without loss of generality that $\varphi(0)=0$. On can observe that for all $\lambda$ large enough,
\begin{equation}
\label{INEG-LAPLACE++1}
    \int_{a}^{b} e^{-\lambda \varphi (u)} f(0) du = \frac{f(0)}{\sqrt{\lambda}}\int_{ a\sqrt{\lambda}}^{b\sqrt{\lambda}} 
    \exp\Big(\!\!-\lambda \varphi\Big(\frac{u}{\sqrt{\lambda}}\Big)\Big)du.
\end{equation}
However, it follows from the assumptions on the function $\varphi$ that
$$\lim_{\lambda \to +\infty} \lambda  \varphi \Big(\frac{u}{\sqrt{\lambda}}\Big) = \varphi''(0)\frac{u^2}{2},$$ 
together with
$$\Big|\exp\Big(\!\!-\lambda \varphi\Big(\frac{u}{\sqrt{\lambda}}\Big)\Big) \Big| \leq \exp(-Cu^2).$$
Consequently, according to the dominated convergence theorem, we obtain that
\begin{equation}
\label{INEG-LAPLACE++2}
\lim_{\lambda \to +\infty} \int_{a\sqrt{\lambda}}^{b\sqrt{\lambda}}  \exp\Big(\!\!-\lambda \varphi\Big(\frac{u}{\sqrt{\lambda}}\Big)\Big) du =\int_{-\infty}^{+\infty} \exp\Big(\! -\varphi''(0)\frac{u^2}{2}\Big)du = \frac{\sqrt{2\pi}}{\sqrt{\varphi''(0)}}.
\end{equation}
Furthermore, by using the usual Laplace method, we find that
\begin{equation}
\label{INEG-LAPLACE++3}
    \Big| \int_{a}^{b} e^{-\lambda \varphi (u)} (f(u)-f(0)) du \Big| \leq 
    \int_{a}^{b} e^{-\lambda \Re \varphi (u)} |f(u)-f(0)| du = o(\lambda^{-1/2}).
\end{equation}
Finally, \eqref{INEG-LAPLACE++1}, \eqref{INEG-LAPLACE++2} and \eqref{INEG-LAPLACE++3} allow to conclude the proof of Lemma \ref{LEMLaplace++}.
\end{proof}

\begin{proof}[Proof of Theorem \ref{T-SLDP}]
We shall now proceed to the proof of Theorem \ref{T-SLDP}. Our goal is to estimate, for all
$x \in ]1/2,1[$, the probability
\begin{eqnarray*}
\dP\Big(\frac{D_n}{n} \geq x \Big) = \sum_{k=\lceil nx \rceil }^{n-1} \dP\left(D_n =k\right).
\end{eqnarray*}
To do so, we extend the Laplace transform $m_n$ of $D_n$, defined in \eqref{Def-Laplace}, into an analytic function on the complex plane. 
For all $t,v\in \dR$, we  have
$$
m_n(t+iv) = \dE\Big[e^{(t+iv) D_n} \Big]= \sum_{k=0}^{n-1} e^{(t+iv)k} \dP(D_n =k).
$$
Therefore, for all $t,v\in \dR$ and for all $k \geq 0$, 
\begin{equation}\label{COEFF-FOURIER}
    \dP\left(D_n =k\right) =e^{-tk}  \frac{1}{2\pi} \int_{-\pi}^{\pi} m_n(t+iv)e^{-ik v}dv.
\end{equation}
One can observe that \eqref{COEFF-FOURIER} is also true for $k \geq n$ and allows us to recover that $\dP(D_n=k)=0$.
Consequently,  {since $|m_n(t+iv)| \leq e^{t n}$,} it follows from Fubini's theorem that for all $t>0$,
\begin{equation}\label{PROBAGEQX}
   \dP\Big(\frac{D_n}{n} \geq x \Big) =    \frac{1}{2\pi} \int_{-\pi}^{\pi} m_n(t+iv)   \sum_{k = \lceil nx \rceil }^{+\infty}   e^{-k(t+i v)}dv.
\end{equation}
In all the sequel, we choose $t=t_x$. In particular, $t_x > 0$ since $x> 1/2$. Then, we deduce from \eqref{PROBAGEQX} that
\begin{equation}
\label{DecIn}
    \dP\Big(\frac{D_n}{n} \geq x \Big) = \frac{1}{2\pi} \int_{-\pi}^{\pi} m_n(t_x+iv)  \frac{\exp( - t_x \lceil nx \rceil  -  i \lceil nx \rceil v) }{1-     e^{-(t_x+i v)}}dv = I_n
\end{equation}
where the integral $I_n$ can be separated into two parts, $I_n=J_n+K_n$ with
\begin{eqnarray*}
J_n & = &  \frac{1}{2\pi} \int_{|v|<\pi - \varepsilon_n} m_n(t_x+iv)  \frac{\exp\left( - t_x \lceil nx \rceil  -  i \lceil nx \rceil v \right) }{1- e^{-(t_x+i v)}}dv, \\ 
K_n & = & \frac{1}{2\pi} \int_{\pi - \varepsilon_n <|v| < \pi } m_n(t_x+iv)  \frac{\exp\left( - t_x \lceil nx \rceil  -  i \lceil nx \rceil v \right) }{1-     e^{-(t_x+i v)}}dv,
\end{eqnarray*}
where $\veps_n= n^{-3/4}$. On the one hand, we obtain from \eqref{MNT} that
\begin{eqnarray*}
    |K_n| &\leq& \frac{1}{2\pi} \int_{\pi - \varepsilon_n <|v| < \pi } |m_n(t_x)|  \frac{\exp\left( - t_x \lceil nx \rceil  \right) }{|1-e^{-t_x}|    }dv,\\
    & \leq &\frac{2\veps_n}{2\pi t_x} \exp(-t_x(\lceil nx \rceil -nx)) \exp(-nI(x))(1+|r_n(t_x)|),\\
    & \leq & \frac{\veps_n}{\pi t_x} \exp(-nI(x))(1+|r_n(t_x)|)
\end{eqnarray*}
by using \eqref{MNT}, \eqref{RATEFI} and the fact that $t_x>0$. Consequently, \eqref{UBRemainder} together
with the definition of $\veps_n$ ensure that
\begin{equation}
\label{Kn}
    K_n =o\Big(\frac{\exp(-nI(x))}{\sqrt{n}}\Big).
\end{equation}
It only remains to evaluate the integral $J_n$. We deduce from \eqref{MNTcompl}
and \eqref{RATEFI} that
\begin{equation} 
    J_n =  \frac{1}{2\pi} \exp( -t_x \{nx \}- n I(x))
           \int_{|v|<\pi - \veps_n} \exp\left( - n \varphi(v) \right)   g_n (v) dv 
\label{Jn}
\end{equation}
where the functions $\varphi$ and $g$ are given, for all $|v| < \pi$, by
$$
\varphi(v) = - \left(L(t_x+ iv) - L(t_x) -i xv  \right)
$$
and
$$
g_n (v)= \frac{ \exp(-i \{nx\}v)} {t_x+i v}  (1 + r_n(t_x+iv)).
$$
Thanks to Lemma \ref{LEM-pte-L}, we have $\varphi(0)=0$, $\varphi'(0)=0$, $\varphi''(0)=\sigma_x^2$ which is a positive real number. 
In addition, there exists a constant $C>0$ such that $\Re \varphi(v) \geq C v^2$.
Therefore, via the extended Laplace method given by Lemma \ref{LEMLaplace++}, we obtain that
\begin{equation}
\label{LaplaceEasy}
\lim_{n \rightarrow \infty} \sqrt{n} \int_{-\pi}^{\pi}  \exp( - n \varphi(v) ) \frac{1}{t_x+iv} dv
=\frac{\sqrt{2\pi}}{ \sigma_x t_x}.
\end{equation}
Hereafter, it follows from \eqref{UBComplexRemainder} that there exists positive constant $a_x, b_x, c_x$ such that for all $|v|\leq \pi-\veps_n$,
\begin{equation*}
|r_n(t_x+iv)|\leq \frac{a_x}{\veps_n }  \Big( 1 - b_x \veps_n\Big) ^{n/2} \leq  \frac{a_x}{\veps_n } \exp \Big( - \frac{b_x n \veps_n}  {2}\Big) 
\leq c_x \exp\Big( - \frac{b_x}{4} n^{1/4} \Big),
\end{equation*}
which ensures that
$$
\big|\exp(-i \{nx\}v)  (1 + r_n(t_x+iv))-1\big|
\leq c_x\exp\Big( - \frac{b_x}{4} n^{1/4} \Big) +|v|.
$$
Then, we obtain that
$$
\left|\int_{-\pi}^{\pi} \exp\left( - n \varphi(v) \right)  \Big( g_n (v)\indicatrice_{|v| <\pi - \veps_n} - \frac{1}{t_x+iv} \Big)dv \right| 
\leq \Lambda_n
$$ 
where 
$$
\Lambda_n=\int_{-\pi}^{\pi} \exp( - n \Re \varphi(v))  \frac{1}{|t_x+iv|}  \Big( c_x\exp\Big( - \frac{b_x}{4} n^{1/4} \Big) +2|v|\Big) dv 
$$
since $\indicatrice_{|v|\geq\pi -\veps_n} \leq |v|$. By the standard Laplace method,
$$
\lim_{n \rightarrow \infty} \sqrt{n} \Lambda_n=0.
$$
Consequently, we deduce from \eqref{Jn} and \eqref{LaplaceEasy} that
\begin{equation}
\label{LimJn}
\lim_{n \rightarrow \infty} \sqrt{n} \exp( t_x \{nx \} + n I(x))J_n
=\frac{1}{\sigma_x t_x\sqrt{2\pi}}.
\end{equation}
Finally, \eqref{DecIn} together with \eqref{Kn} and \eqref{LimJn} clearly lead to \eqref{SLDPR}.
\end{proof}

\section{An alternative proof.}
\label{S-PRU}
We already saw from \eqref{Tanny} that the distribution $D_n$ is nothing else than the one of the integer part of the Irwin-Hall distribution
which is the sum $S_n=U_1+\cdots+U_n$ of independent and identically distributed random variables sharing the same uniform distribution on $[0,1]$.
It follows from some direct calculation that for any $x\in ]1/2, 1[$,
\begin{eqnarray}
    \dP\Big(\frac{D_n}{n}\geq x\Big) &=& \dP( S_n \geq \lceil nx \rceil), \nonumber\\
     &=&  \dE_n \Big[ \exp(- t_x S_n + nL(t_x)) \ind_{\{\frac{S_n}{n} \geq x +\veps_n \}} \Big ], \nonumber\\
    &=&\exp(-nI(x) ) \dE_n \Big[ e^{-nt_x \left(\frac{S_n}{n} -x \right)}\ind_{\{\frac{S_n}{n} \geq x+\veps_n \}} \Big ],
    \label{change-proba}
\end{eqnarray}
where $\dE_n$ is the expectation under the new probability $\dP_n$ given by
\begin{equation}
\label{NewP}
 \frac{d\dP_n}{d\dP}= \exp\big(t_x S_n -n L(t_x)\big)
\end{equation}
and $\veps_n= \{nx\}/n$. Let
$$V_n= \frac{\sqrt{n}}{\sigma_x}  \Big( \frac{S_n}{n}-x \Big)$$
and denote $f_n$ and $\Phi_n$ the probability density function and the 
characteristic function of $V_n$ under the new probability $\dP_n$, respectively. Let us remark that, under $\dP_n$, we know that $(V_n)$ converges in distribution to the standard Gaussian measure. Using Parseval identity, we have 
\begin{eqnarray}
    \nonumber\dE_n \Big[ e^{-nt_x \left(\frac{S_n}{n} -x \right)}\ind_{\{\frac{S_n}{n} \geq x+\veps_n \}} \Big ]
     \nonumber&=& \int_{\dR} e^{-\sigma_x t_x \sqrt{n} v} \ind_{\{ v \geq \frac{\sqrt n \veps_n}{\sigma_x}\} }
         f_{n}(v) dv,\\
          \nonumber &=& \frac{1}{2\pi} \int_{\dR} \frac{e^{-(\sigma_x t_x\sqrt n+iv) \frac{\sqrt n \veps_n}{\sigma_x} }} {\sigma_x t_x \sqrt n+iv} \Phi_n(v) dv,\\
          &=& \frac{e^{-t_x\{nx\}}}{2\pi\sigma_x \sqrt{n}} \int_{\dR} \frac{e^{-i \frac{ \{nx\} v}{\sigma_x \sqrt n} }} {t_x+i\frac{v}{\sigma_x \sqrt n}} \Phi_n(v) dv.
          \label{Int-Parseval}
\end{eqnarray}
Recalling that $L$ is also the logarithm of the Laplace transform of the uniform distribution in $[0,1]$, we
obtain from \eqref{NewP} that
 \begin{eqnarray*}
 \Phi_n(v)&=&\dE \Big[ \exp\Big( \frac{iv S_n}{\sigma_x \sqrt n} - \frac{i\sqrt n x v}{\sigma_x} +t_x S_n -nL(t_x)\Big) \Big],\\
          &=&\exp\Big( n \Big( L\big(t_x+ \frac{iv}{\sigma_x \sqrt n}\big) -L(t_x)- \frac{i x v}{ \sqrt n\sigma_x} \Big) \Big).
 \end{eqnarray*}
 Let $A$ be a positive constant chosen later. We can split the integral in \eqref{Int-Parseval} into two parts: we call $J_n$ the integral on $[-A\sigma_x\sqrt{n}, +A\sigma_x\sqrt{n}]$ and $K_n$ the integral on the complementary set.
On the one hand,   {as \eqref{MAJORATION} also holds replacing $\pi^2$ by $v^2$ which is smaller than $A^2 \sigma_x^2 n$ and since $L''(t_x)=\sigma_x^2$,} we get that for all $v \in \dR$,
\begin{eqnarray}
    \left| \frac{e^{-i \frac{ \{nx\} v}{\sigma_x \sqrt n} }} {t_x+i\frac{v}{\sigma_x \sqrt n}} \Phi_n(v) \ind_{\{|v| \leq A \sigma_x\sqrt{n}\}} \right| &\leq&  \frac{1}{t_x} \left(1-\frac{t_x^2 \sigma_x^2}{t_x^2+A^2}\frac{v^2}{\sigma_x^2 n} \right)^{n/2}, \nonumber \\
    &\leq& \frac{1}{t_x} \exp \left( -\frac{t_x^2}{t_x^2+A^2}\frac{v^2}{2} \right).
    \label{majoration-domination}
\end{eqnarray}
Then, we deduce from the Lebesgue's dominated convergence theorem that
\begin{equation}
\label{limite-Jn}
    \lim_{n \rightarrow + \infty} J_n = \int_{\dR} \frac{1}{t_x} \exp \left( -\frac{L''(t_x)}{\sigma_x^2} \frac{v^2}{2} \right) dv = \frac{\sqrt{2\pi}}{t_x}. 
\end{equation}
 On the other hand, concerning $K_n$, since now $v$ is large in the integral, we use \eqref{MAJORATION-v-grand} to get that
$$
    \left| \frac{e^{-i \frac{ \{nx\} v}{\sigma_x \sqrt n} }} {t_x+i\frac{v}{\sigma_x \sqrt n}} \Phi_n(v)  \right|\leq \frac{1}{t_x} \left( \frac{t_x^2+4 t_x^2 \frac{e^{t_x}}{(e^{t_x}-1)^2}}{t_x^2+\frac{v^2}{\sigma_x^2 n}} \right)^{n/2}
$$
leading to
 \begin{equation}
     K_n \leq \frac{2}{t_x}\left( t_x^2+4 t_x^2 \frac{e^{t_x}}{(e^{t_x}-1)^2}  \right)^{n/2} \sigma_x \sqrt{n} \int_{A}^{+\infty}  \frac{1}{(t_x^2 +v^2)^{n/2}} dv.
     \label{borne-Kn}
 \end{equation}
 Moreover, for all $n>2$,
 \begin{equation}
     \int_{A}^{+\infty}  \frac{1}{(t_x^2 +v^2)^{n/2}} dv \leq \int_{A}^{+\infty}  \frac{v}{A}\frac{1}{(t_x^2 +v^2)^{n/2}} dv   = \frac{1}{(n-2)A} \frac{1}{(t_x^2+ A^2)^{n/2-1}}.
     \label{borne-integrale-Kn}
 \end{equation}
By taking 
$$A^2 = t_x^2+8 t_x^2  \frac{e^{t_x}}{(e^{t_x}-1)^2},$$ 
we obtain from \eqref{borne-Kn} and \eqref{borne-integrale-Kn} that 
\begin{equation}
\label{limite-Kn}
    \lim_{n \rightarrow + \infty} K_n = 0
\end{equation}
exponentially fast. Finally, \eqref{change-proba} together with \eqref{Int-Parseval}, \eqref{limite-Jn} and
\eqref{limite-Kn} allow us to conclude the alternative proof of Theorem \ref{T-SLDP}.
\demend
 
\begin{rem}
\label{rem-concentration-direct}
We can use previous computations to get also a new concentration inequality. More precisely, by using the upper-bound \eqref{majoration-domination} in order to get an upper bound for $J_n$ instead of a limit, we are able to prove that
\begin{equation}
\label{NEWCID}
    \dP\Big(\frac{D_n}{n} \geq x \Big) \leq  Q_n(x)\frac{\exp(-n I(x)-\{nx\}t_x)}{\sigma_x t_x \sqrt{2\pi n}}
\end{equation}
where the prefactor can be taken as
\begin{equation*}
    Q_n(x) =  \sqrt{2+\frac{8 e^{t_x}}{(e^{t_x}-1)^2}}+ \frac{4\sigma_x t_x}{\sqrt{2\pi}}  \sqrt{1+\frac{8e^{t_x}}{(e^{t_x}-1)^2}}\frac{\sqrt{n}}{2^{n/2}(n-2)} .
\end{equation*}
One can observe that \eqref{NEWCID} is similar to \eqref{CID}. 
Note that the constants in \eqref{CID} as well as in \eqref{NEWCID} are not sharp. 
It is in fact possible to improve them by  more precise cuttings in the integrals.
\end{rem}


\section{Proof of the concentration inequalities}
\label{S-CI}
We shall now proceed to the proof of the concentration inequalities.
Recalling that $x \in ]1/2,1[$ which implies that $t_x>0$, we obtain from equality \eqref{PROBAGEQX} that
\begin{align*}
   \dP\Big(\frac{D_n}{n} \geq x \Big)& =    \frac{1}{2\pi} \int_{-\pi}^{\pi} m_n(t_x+iv)   \sum_{k = \lceil nx \rceil }^{+\infty}   e^{-k(t_x+i v)}dv,\\
   &=
 \frac{1}{2\pi} \int_{-\pi}^{\pi} m_n(t_x+iv)   \frac{\exp\left( - t_x \lceil nx \rceil  -  i \lceil nx \rceil v \right) }{1-     e^{-(t_x+i v)}}dv, \\
 &\leq  \frac{ \exp(-\lceil n x\rceil t_x) }{2\pi} \int_{-\pi}^{\pi} \frac{|m_n(t_x+iv)|}{|1-     e^{-(t_x+i v)}|} dv.
\end{align*} 
Consequently, we deduce from the alternative upper-bound \eqref{UBComplexRemainder2} that
\begin{equation}
    \label{PRCI1}
    \dP\Big(\frac{D_n}{n} \geq x \Big) \leq \frac{1}{2\pi}  \exp\Big(\! - n (xt_x-L(t_x)) -\{nx\}t_x \Big) \big(A(x)+B(x)\big) 
\end{equation} 
where 
\begin{eqnarray*}
    A(x) &=& \int_{-\pi}^{\pi} \frac{1}{\sqrt{t_x^2+ v^2}} \exp \Big(-n \frac{t_x^2 L''(t_x)}{t_x^2+\pi^2} \frac{v^2}{2} \Big)dv, \\
    B(x) &=& \left(1+\frac1\pi+\frac{ 2\sqrt{t_x^2+\pi^2} }{\pi^2 -4} \right) \int_{-\pi}^{\pi} \exp \Big(\!\! -n \frac{4t_x^2 L''(t_x)}{\pi^2(t_x^2+4)} \frac{v^2}{2} \Big) dv.
\end{eqnarray*}
Hereafter, we recall from \eqref{DEFI} that $I(x)=x t_x -L(t_x)$ and we denote $\sigma_x^2=L''(t_x)$. It follows from standard Gaussian calculation that 
\begin{eqnarray}
\label{CIA}
    A(x) & \leq &  \frac{2 \pi}{ \sigma_x t_x \sqrt{2\pi n}}\sqrt{\frac{t_x^2+\pi^2}{t_x^2 }}, \\
    B(x) & \leq & \left(1+\frac1\pi+\frac{ 2\sqrt{t_x^2+\pi^2} }{\pi^2 -4} \right) \frac{\pi^2 \sqrt{t_x^2+4}}{ \sigma_x t_x \sqrt{2\pi n}}.
\label{CIB}
\end{eqnarray}
Finally, we find from \eqref{PRCI1} together with \eqref{CIA} and \eqref{CIB} that
\begin{equation*}
    \dP\Big(\frac{D_n}{n} \geq x \Big) \leq  P(x)\frac{\exp(-n I(x)-\{nx\}t_x)}{\sigma_x t_x \sqrt{2\pi n}}
\end{equation*}
where 
\begin{align*}
    P(x) =  \sqrt{\frac{t_x^2 + \pi^2}{t_x^2}}+ \left(1+\frac1\pi+\frac{ 2\sqrt{t_x^2+\pi^2} }{\pi^2 -4} \right) \sqrt{\frac{\pi^2(t_x^2+4)}{4}}   ,
\end{align*}
which is exactly what we wanted to prove.


\section{Proof of the standard results}
\label{S-SR}

We now focus our attention on the more standard results concerning the sequence $(D_n)$ such as the quadratic strong law, 
the law of iterated logarithm and the functional central limit theorem.

\begin{proof}[Proof of Proposition \ref{P-QSLLIL}]
First of all, one can observe from \eqref{DECDN} and \eqref{DEFMN} that the martingale $(M_n)$ can be rewritten in the additive form
\begin{equation}
\label{NEWDEFMN}
M_n=\sum_{k=1}^{n-1} (k+1) (\xi_{k+1}-p_k).
\end{equation}
It follows from the almost sure convergence \eqref{ASCVGDN} together with \eqref{PQVMN} and the classical Toeplitz lemma that
the predictable quadratic variation $\langle M \rangle_n$ of $(M_n)$ satisfies
\begin{equation}
\label{CVGPQVMN}
\lim_{n \rightarrow \infty}\frac{\langle M \rangle_n}{n^3}=\frac{1}{12} \hspace{1cm} \text{a.s.}
\end{equation}
Denote by $f_n$ the explosion coefficient associated with $(M_n)$,
\begin{equation}
f_n=\frac{\langle M \rangle_{n+1}-\langle M \rangle_n}{\langle M \rangle_{n+1}}=\frac{(n-D_n)(D_n+1)}{\langle M \rangle_{n+1}}. \nonumber
\end{equation}
We obtain from \eqref{ASCVGDN} and \eqref{CVGPQVMN} that
\begin{equation}
\label{CVGFN}
\lim_{n \rightarrow \infty}nf_n=3 \hspace{1cm} \text{a.s.}
\end{equation}
which implies that $f_n$ converges to zero almost surely as $n$ goes to infinity. 
In addition, we clearly have for all $n\geq 1$, $| \xi_{n+1} -p_n | \leq 1$. Consequently, we deduce from the quadratic strong law for martingales given e.g. by Theorem 3 in \cite{Bercu2004} that
\begin{equation*}
\lim_{n \rightarrow \infty} \frac{1}{\log \langle M \rangle_n}\sum_{k=1}^n f_k \frac{M_k^2}{\langle M \rangle_k}=1 \hspace{1cm}\text{a.s.}
\end{equation*}
which ensures that
\begin{equation}
\label{MGQSL1}
\lim_{n \rightarrow \infty} \frac{1}{\log n}\sum_{k=1}^n  \frac{M_k^2}{k^4}=\frac{1}{12} \hspace{1cm}\text{a.s.}
\end{equation}
However, it follows from \eqref{DEFMN} that
\begin{equation}
\label{MGQSL2}
\frac{M_n^2}{n^4}=\Big(\frac{D_n}{n}-\frac{1}{2}\Big)^2+\frac{1}{n}\Big(\frac{D_n}{n}-\frac{1}{2}\Big)+\frac{1}{4n^2}
\end{equation}
Therefore, we obtain once again from \eqref{ASCVGDN} together with \eqref{MGQSL1} and \eqref{MGQSL2} that
\begin{equation*}
\lim_{n \rightarrow \infty} \frac{1}{\log n}\sum_{k=1}^n  \Big(\frac{D_k}{k}-\frac{1}{2}\Big)^2=\frac{1}{12} \hspace{1cm}\text{a.s.}
\end{equation*}
which is exactly the quadratic strong law \eqref{QSL}. It only remains to prove the law of iterated logarithm given by \eqref{LIL}.
It immediately follows from the law of iterated logarithm for martingales given e.g. by Corollary 6.4.25 in \cite{Duflo1997} that 
\begin{eqnarray*}
\limsup_{n \rightarrow \infty}  \Bigl(\frac{1}{2  \langle M \rangle_n \log \log \langle M \rangle_n}\Bigr)^{1/2}M_n 
&=& -\liminf_{n \rightarrow \infty}  \Bigl(\frac{1}{2  \langle M \rangle_n \log \log \langle M \rangle_n }\Bigr)^{1/2} M_n 
\notag\\
&=& 1 \hspace{1cm}\text{a.s.}
\end{eqnarray*}
which leads via \eqref{CVGPQVMN} to
\begin{eqnarray}
\limsup_{n \rightarrow \infty}  \Bigl(\frac{1}{2  n^3 \log \log  n}\Bigr)^{1/2}M_n 
&=& -\liminf_{n \rightarrow \infty}  \Bigl(\frac{1}{2  n^3 \log \log  n }\Bigr)^{1/2} M_n 
\notag\\
&=& \frac{1}{\sqrt{12}} \hspace{1cm}\text{a.s.}
\label{LILMG}
\end{eqnarray}
Finally, 
we find from \eqref{DEFMN} and \eqref{LILMG} that
\begin{eqnarray*}
\limsup_{n \rightarrow \infty}  \Bigl(\frac{n}{2 \log \log n}\Bigr)^{1/2} \Bigl(\frac{D_n}{n}-\frac{1}{2}\Bigr) 
&=& -\liminf_{n \rightarrow \infty}  \Bigl(\frac{n}{2 \log \log n}\Bigr)^{1/2} \Bigl(\frac{D_n}{n}-\frac{1}{2}\Bigr) 
\notag\\
&=& \frac{1}{\sqrt{12}} \hspace{1cm}\text{a.s.}
\end{eqnarray*}
which achieves the proof of Proposition \ref{P-QSLLIL}.
\end{proof}

\begin{proof}[Proof of Proposition \ref{P-FCLT}]
We shall now proceed to the proof of the functional central limit theorem 
given by the distributional convergence \eqref{FCLT}. On the one hand, it follows from \eqref{CVGPQVMN} that
for all $t \geq 0$, 
\begin{equation}
\lim_{n\rightarrow \infty} \frac{1}{n^3}\langle M \rangle_{\lfloor nt \rfloor} =\frac{t^{3}}{12} \hspace{1cm}\text{a.s.}
\label{CVGPQVT}
\end{equation}
On the other hand, it is quite straightforward to check that $(M_n)$ satisfies Lindeberg's condition given, for all $t \geq 0$ and for any $\varepsilon>0$, by
\begin{equation}
\label{LINDEBERGT}
\frac{1}{n^{3}}\sum_{k=2}^{\lfloor nt \rfloor}\dE\big[\Delta M_k^2 \rI_{\{|\Delta M_k|>\varepsilon \sqrt{n^{3}} \}}|\cF_{k-1}\big] \limp 0
\end{equation}
where $\Delta M_n=M_n-M_{n-1}=n(\xi_n - p_{n-1})$. As a matter of fact, we have for all $t \geq 0$ and for any $\varepsilon>0$,
\begin{equation*}
 \frac{1}{n^{3}}\sum_{k=2}^{\lfloor nt \rfloor}\dE\big[\Delta M_k^2 \rI_{\{|\Delta M_k|>\varepsilon \sqrt{n^{3}} \}}|\cF_{k-1}\big] 
\leq 
\frac{1}{n^{6} \varepsilon^2} \sum_{k=2}^{\lfloor nt \rfloor} \dE\big[\Delta M_k^4|\cF_{k-1}\big].
\end{equation*}
However, we already saw that for all
$n \geq 2$, $|\Delta M_n|\leq n$. Consequently, we obtain that for all $t \geq 0$ and for any $\varepsilon>0$,
\begin{equation*}
 \frac{1}{n^{3}}\sum_{k=2}^{\lfloor nt \rfloor}\dE\big[\Delta M_k^2 \rI_{\{|\Delta M_k|>\varepsilon \sqrt{n^{3}} \}}|\cF_{k-1}\big] 
\leq 
\frac{1}{n^{6} \varepsilon^2} \sum_{k=2}^{\lfloor nt \rfloor} k^4 \leq
\frac{t^5}{n \varepsilon^2}
\end{equation*}
which immediately implies \eqref{LINDEBERGT}.
Therefore, we deduce from \eqref{CVGPQVT} and \eqref{LINDEBERGT} together with the functional central limit theorem  for martingales given e.g. in Theorem 2.5 of \cite{Durrett1978} that
\begin{equation}
\label{FCLTMART}
\Big(\frac{M_{\lfloor nt \rfloor}}{\sqrt{n^{3}}}, t \geq 0\Big) \Longrightarrow \big( B_t, t \geq 0 \big)
\end{equation}
where $\big( B_t, t \geq 0 \big)$ is a real-valued centered Gaussian process starting at the origin with covariance given, for all $0<s \leq t$, by
$\dE[B_s B_t]= s^{3}/12$.
Finally, \eqref{DEFMN} and \eqref{FCLTMART} lead to
\eqref{FCLT} where $W_t=B_t/t^2$,
which is exactly what we wanted to prove.
\end{proof}

\section*{Acknowledgment}
The authors would like to thank the two anonymous reviewers and the associate editor for their careful reading and helpful comments
which help to improve the paper substantially.

\bibliographystyle{abbrv}
\bibliography{Biblio-LDPDES}

\begin{thebibliography}{10}

\bibitem{Bercu2004}
B.~Bercu.
\newblock On the convergence of moments in the almost sure central limit
  theorem for martingales with statistical applications.
\newblock {\em Stochastic Process. Appl.}, 111(1):157--173, 2004.

\bibitem{Bercu2015}
B.~Bercu, B.~Delyon, and E.~Rio.
\newblock {\em {Concentration Inequalities for Sums and Martingales}}.
\newblock SpringerBriefs in Mathematics. {Springer}, 2015.

\bibitem{Bona2012}
M.~B\'{o}na.
\newblock {\em Combinatorics of permutations}.
\newblock Discrete Mathematics and its Applications. CRC Press, Boca Raton, FL,
  second edition, 2012.

\bibitem{Bryc2009}
W.~Bryc, D.~Minda, and S.~Sethuraman.
\newblock Large deviations for the leaves in some random trees.
\newblock {\em Adv. in Appl. Probab.}, 41(3):845--873, 2009.

\bibitem{Chatterjee2017}
S.~Chatterjee and P.~Diaconis.
\newblock A central limit theorem for a new statistic on permutations.
\newblock {\em Indian J. Pure Appl. Math.}, 48(4):561--573, 2017.

\bibitem{Duflo1997}
M.~Duflo.
\newblock {\em Random iterative models}, volume~34 of {\em Applications of
  Mathematics (New York)}.
\newblock Springer-Verlag, Berlin, 1997.

\bibitem{Durrett1978}
R.~Durrett and S.~I. Resnick.
\newblock Functional limit theorems for dependent variables.
\newblock {\em Ann. Probab.}, 6(5):829--846, 1978.

\bibitem{Flajolet2005}
P.~Flajolet, J.~Gabarró, and H.~Pekari.
\newblock Analytic urns.
\newblock {\em Ann. Probab.}, 33(3):1200--1233, 05 2005.

\bibitem{Freedman1965}
D.~A. Freedman.
\newblock Bernard {F}riedman's urn.
\newblock {\em Ann. Math. Statist.}, 36:956--970, 1965.

\bibitem{Friedman1949}
B.~Friedman.
\newblock A simple urn model.
\newblock {\em Comm. Pure Appl. Math.}, 2:59--70, 1949.

\bibitem{Garet2021}
O.~Garet.
\newblock A central limit theorem for the number of descents and some urn
  models.
\newblock {\em Markov Process. Related Fields}, 27(5):789--801, 2021.

\bibitem{Gnedin2006}
A.~Gnedin and G.~Olshanski.
\newblock The boundary of the {E}ulerian number triangle.
\newblock {\em Mosc. Math. J.}, 6(3):461--475, 587, 2006.

\bibitem{Gouet1993}
R.~Gouet.
\newblock Martingale functional central limit theorems for a generalized
  {P}\'{o}lya urn.
\newblock {\em Ann. Probab.}, 21(3):1624--1639, 1993.

\bibitem{Houdre2010}
C.~Houdr\'{e} and R.~Restrepo.
\newblock A probabilistic approach to the asymptotics of the length of the
  longest alternating subsequence.
\newblock {\em Electron. J. Combin.}, 17(1):Research Paper 168, 19, 2010.

\bibitem{Ozdemir2021}
A.~\"{O}zdemir.
\newblock Martingales and descent statistics.
\newblock {\em Adv. in Appl. Math.}, 140:Paper No. 102395, 34, 2022.

\bibitem{Stanley2008}
R.~P. Stanley.
\newblock Longest alternating subsequences of permutations.
\newblock {\em Michigan Math. J.}, 57:675--687, 2008.
\newblock Special volume in honor of Melvin Hochster.

\bibitem{Tanny1973}
S.~Tanny.
\newblock A probabilistic interpretation of {E}ulerian numbers.
\newblock {\em Duke Math. J.}, 40:717--722, 1973.

\bibitem{Widom2006}
H.~Widom.
\newblock On the limiting distribution for the length of the longest
  alternating sequence in a random permutation.
\newblock {\em Electron. J. Combin.}, 13(1):Research Paper 25, 7, 2006.

\bibitem{Zhang2015}
Y.~Zhang.
\newblock On the number of leaves in a random recursive tree.
\newblock {\em Braz. J. Probab. Stat.}, 29(4):897--908, 2015.

\bibitem{Zwillinger1989}
D.~Zwillinger.
\newblock {\em Handbook of differential equations}.
\newblock Academic Press, Inc., Boston, MA, 1989.

\end{thebibliography}

\end{document}